\documentclass[12pt]{amsart}
\usepackage{amssymb,amsbsy}
\usepackage{amscd}
\usepackage{enumerate}
\usepackage{setspace}
\usepackage{tikz}
\usepackage{bm}
\usepackage{cite}
\usepackage{graphicx}
\usepackage{hyperref}
\usepackage{nicematrix}
\usepackage{colortbl}
\allowdisplaybreaks

\usepackage{kpfonts}
\usepackage{stackengine}
\usepackage{calc}
\newlength\shlength
\newcommand\xshlongvec[2][0]{\setlength\shlength{#1pt}%
  \stackengine{-5.6pt}{$#2$}{\smash{$\kern\shlength%
    \stackengine{7.55pt}{$\mathchar"017E$}%
      {\rule{\widthof{$#2$}}{.57pt}\kern.4pt}{O}{r}{F}{F}{L}\kern-\shlength$}}%
      {O}{c}{F}{T}{S}}
  
\usepackage[letterpaper]{geometry}
\theoremstyle{definition}
\newtheorem{theorem}{Theorem}[section]
\newtheorem{lemma}[theorem]{Lemma}
\newtheorem{proposition}[theorem]{Proposition}
\newtheorem{definition}[theorem]{Definition}
\newtheorem{corollary}[theorem]{Corollary}

\newtheorem*{remark*}{Remark}

\newtheorem{example}[theorem]{Example}

\newcommand\N{\mathbb N}
\newcommand\K{\mathbb K}

\newcommand\ce{r}

\newcommand\Atc{A^{\pm}}
\newcommand\Atcp{A^+}
\newcommand\Atcm{A^-}

\title[Matrices similar to centrosymmetric matrices]{Matrices similar to centrosymmetric matrices}
\author{Benjam\'in A. Itz\'a-Ortiz}
\author{Rub\'en A. Mart\'inez-Avenda\~no}

\address{Centro de Investigaci\'on en Matem\'aticas, Universidad Aut\'onoma del Estado de Hidalgo, Pachuca, Hidalgo, Mexico}
\address{Departamento Acad\'emico de Matem\'aticas, Instituto Tecnol\'ogico Aut\'onomo de M\'exico, Mexico City, Mexico}

\email{itza@uaeh.edu.mx}
\email{ruben.martinez.avendano@gmail.com}

\thanks{\\
This paper was motivated by the factorization of determinants used in \cite{IMN2}, which arose from an idea by Prof. Hiroshi Nakazato. The authors would like to thank Prof. Nakazato for sharing his insights with us and for all the help provided throughout the investigation that resulted in this paper.
\\
The authors are thankful to the reviewer for his/her careful reading of this paper and for the comments, which helped to improve the manuscript. 
\\
The second author's research is partially supported by the Asociaci\'on Mexicana de Cultura A.C}


\begin{document}
\begin{abstract}
In this paper we give conditions on a matrix which guarantee that it is similar to a centrosymmetric matrix. We use this conditions to show that some $4 \times 4$ and $6 \times 6$ Toeplitz matrices are similar to centrosymmetric matrices. Furthermore, we give conditions for a matrix to be similar to a matrix which has a centrosymmetric principal submatrix, and conditions under which a matrix can be dilated to a matrix similar to a centrosymmetric matrix.
\end{abstract}

\maketitle

\setlength\arraycolsep{2pt}

\section*{Introduction}

A matrix is {\em centrosymmetric} if it is ``symmetric about its center'' \cite{weaver}; that is, if it remains unchanged if we reflect it horizontally and vertically (see Definition~\ref{def:centro}). Although the term ``centrosymmetric'' was first introduced by Aitken in his comprehensive study of determinants~\cite{aitken}, 
the study of these type of matrices seems to have started with the study of the so-called Sylvester-Kac determinant (see \cite{TaTo}). Centrosymmetric matrices appear naturally in many places and have several applications (see the list given in the introduction of \cite{CaBu} for some of them), particularly in the study of Markov processes (see, for example, \cite{weaver}). Centrosymmetric matrices have many interesting properties: we refer the reader to the papers of Abu-Jeib~\cite{abu}, Cantoni and Butler~\cite{CaBu}, Good~\cite{good} and Weaver~\cite{weaver} for some of them.

Since many matrix-theoretic properties are preserved under similarity, finding conditions under which a matrix is similar to a centrosymmetric matrix is a natural question. In this paper, we show that if we write a matrix $M$ as a block-matrix and a system of matrix equations built from these blocks admits an invertible solution, then the matrix $M$ is similar to a centrosymmetric matrix. Of course, if the matrix $M$ itself is centrosymmetric, it is well-known that a solution to this system is a concrete permutation matrix; it turns out that whenever this permutation matrix is a solution, the matrix $M$ itself is centrosymmetric, a known fact (see, e.g. \cite{aitken}). Furthermore, in this paper we show that if the solution to this system of matrix equations is not invertible, then, depending on the rank of $M$, we have that $M$ is similar to a matrix containing a centrosymmetric principal submatrix or $M$ can be dilated to a matrix which is similar to a centrosymmetric matrix.

It is worth observing that a solution to the system of equations mentioned above is also a solution to an algebraic matrix Riccati equation (see, for example, \cite{Meyer} for a discussion of the algebraic matrix Riccati equation). But in general, it is not true that having a solution to the algebraic Riccati equation, implies that the matrix $M$ is similar to a centrosymmetric matrix, as we will show later in this paper.

We now describe the contents of this paper. In Section~\ref{se:similar}, after giving a well-known characterization of centrosymmetric matrices, we show that if we form a system of equations with the blocks of the matrix $M$ and this system has an invertible solution, then $M$ is similar to a centrosymmetric matrix; furthermore, in the case where such solution is unitary, then $M$ is unitarily equivalent to a centrosymmetric matrix.
We use this theorem to show that some Toeplitz matrices are similar to centrosymmetric matrices.
We also show that if the solution $X$ is not invertible, then depending on the rank of $X$, the matrix $M$ is similar to a matrix which has a centrosymmetric principal submatrix or the matrix $M$ can be dilated to a matrix which is similar to a centrosymmetric matrix (it may be even unitarily equivalent to one).
 
It is well-known that if the Riccati equation has a solution, then the determinant of the matrix $M$ can be computed as the product of two smaller determinants. Motivated by this fact and by the connection of our results to the Riccati equation, in Section~\ref{se:fact}, we show that the matrix $M$ is singular if there is a nonzero solution to an specific system of equations (which will also be a solution to the Riccati equation). We end this section by posing a question on the uniqueness of the factorization of the determinant of $M$, whenever two different Riccati equations are satisfied.

 We finish by proving, in Section~\ref{se:exa}, two results which are used in \cite{IMN2} and which originally motivated the results in this paper.

\section{Similarity to Centrosymmetric Matrices}\label{se:similar}

Throughout this paper, $\K$ will denote an arbitrary field (not necessarily of characteristic zero) and all matrices will have entries in the field $\K$. As usual, $\N$ will denote the set of positive integers and $I=I_n$ will denote the identity matrix of size $n \times n$.

The following definition is classical (e.g, \cite{aitken}).

\begin{definition}\label{def:centro}
Let $n \in \N$ and let $M$ be an $n\times n$ matrix. We say that $M$ is {\em centrosymmetric} if $m_{i,j}=m_{n-i+1,n-j+1}$ for $i,j=1,2, \dots, n$, where $m_{i,j}$ denotes the $(i,j)$-th entry of $M$.
\end{definition}

We denote by $J_1$ the $1\times 1$ matrix $J_1=\left( 1 \right)$ and for $n \in \N$, with $n \geq 2$, we denote by $J_n$ the $n\times n$ centrosymmetric matrix 
\[
J_n=\begin{pmatrix}
0 & 0 & \cdots & 0 & 1 \\
0 & 0 & \cdots & 1 & 0 \\
\vdots & \vdots & \iddots & \vdots & \vdots \\
0 & 1 & \cdots & 0 & 0 \\
1 & 0 & \cdots & 0 & 0 \\
\end{pmatrix}.
\]

When there is no confusion about the size, we will denote this matrix by $J$, dropping the subindex $n$. It is easy to check that a square matrix $M$ is centrosymmetric if and only if $M J = J M$ (see, for example, \cite[Proposition 6]{weaver}). Using this fact, the following lemma is straightforward (\cite[p.~124]{aitken}).

\begin{lemma}\label{le:char_centrosym}
Let $n \in \N$ and let $M$ be a $n\times n$ matrix. 
\begin{itemize}
    \item Assume $n$ is even. Then $M$ is centrosymmetric if and only if $M$ can be written as a block matrix of the form
\[
M= \left( \begin{array}{c|c}
A & B \\
\hline 
C & D 
\end{array}
\right)
\]
where $A,B,C,D$ are all $\frac{n}{2} \times \frac{n}{2}$ matrices, $J A=D J$ and $C=J B J$, where $J=J_{\frac{n}{2}}$.

\item Assume $n\geq 3$ is odd. Then $M$ is centrosymmetric if and only if $M$ can be written as a block matrix of the form
\[
M= \left( \begin{array}{c|c|c}
A & x & B \\
\hline
y & \mu & z \\
\hline
C & w & D 
\end{array}
\right)
\]
where $A,B,C,D$ are all $\frac{n-1}{2} \times \frac{n-1}{2}$ matrices, where $x, w$ are $\frac{n-1}{2} \times 1$ matrices, where $y, z$ are $1 \times \frac{n-1}{2}$ matrices, $\mu \in \K$, and $J A= D J$,  $C=J B J$, $w = J x$ and $y = z J$, where $J=J_{\frac{n-1}{2}}$.
\end{itemize}
\end{lemma}

If we replace the matrix $J$ in the previous lemma by an arbitrary invertible matrix $X$ we show, in the next theorem, that we obtain matrices similar to centrosymmetric matrices. 

\begin{theorem}\label{th:sim_cen}
Let $n \in \N$ and let $M$ be an $n\times n$ matrix.
\begin{itemize}
    \item Assume $n$ is even and $M$ is written as
    \[
M= \left( \begin{array}{c|c}
A & B \\
\hline 
C & D 
\end{array}
\right),
\]
where $A$, $B$, $C$ and $D$ are $\frac{n}{2} \times \frac{n}{2}$ matrices. Assume there exists an invertible matrix $X$, of size $\frac{n}{2} \times \frac{n}{2}$, such that $XA=DX$ and $C=XBX$. Then $M$ is similar to a centrosymmetric matrix.

\item  Assume $n \geq 3$ is odd and $M$ is written as
\[
M= \left( \begin{array}{c|c|c}
A & x & B \\
\hline
y & \mu & z \\
\hline
C & w & D 
\end{array}
\right)
\]
where $A,B,C,D$ are all $\frac{n-1}{2} \times \frac{n-1}{2}$ matrices, where $x, w$ are $\frac{n-1}{2} \times 1$ matrices, where $y, z$ are $1 \times \frac{n-1}{2}$ matrices, and $\mu \in \K$. Assume there exists an invertible matrix $X$, of size $\frac{n-1}{2} \times \frac{n-1}{2}$, such that $XA=DX$,  $C=XBX$, $w=X x$ and $y=z X$. Then $M$ is similar to a centrosymmetric matrix.
\end{itemize}
\end{theorem}
\begin{proof}
We prove the case where $n$ is even first. Let $Q$ be the block matrix
\[
\left(
\begin{array}{c|c}
I & 0 \\
\hline
0 &  X J
\end{array}
\right),
\]
where $I=I_{\frac{n}{2}}$ and $J=J_{\frac{n}{2}}$. Then 
\[
Q^{-1} M Q = \left(
\begin{array}{c|c}
I & 0 \\ 
\hline
0 &  J X^{-1}
\end{array}
\right) 
\left( \begin{array}{c|c}
A & B \\
\hline 
C & D 
\end{array}
\right)
\left(
\begin{array}{c|c}
I & 0 \\
\hline
0 &  X J
\end{array}
\right)=
\left(
\begin{array}{c|c}
A & B X J \\
\hline
J X^{-1} C & J X^{-1} D X J
\end{array}
\right).
\]
Observe that
\[
XA = DX \implies  A = X^{-1} D X \implies JA = J (X^{-1} D X) \implies JA = (J X^{-1} D X J) J
\]
and 
\[
XBX=C \implies BX = X^{-1} C  \implies J (BX) = J(X^{-1} C)  \implies J (BX J) J = J X^{-1} C
\]
so by Lemma~\ref{le:char_centrosym} we obtain that $Q^{-1} M Q$ is centrosymmetric. 

Now, if $n$ is odd, let $Q$ be the block matrix
\[
\left(
\begin{array}{c|c|c}
I & 0 & 0 \\
\hline
0 & 1 & 0
\\
\hline
0 &  0 & X J
\end{array}
\right),
\]
where $I=I_{\frac{n-1}{2}}$, $J=J_{\frac{n-1}{2}}$ and the central block is the element $1 \in \K$. Then 
\[
\left(
\begin{array}{c|c|c}
I & 0 & 0 \\
\hline
0 & 1 & 0
\\
\hline
0 &  0 & J X^{-1}
\end{array}
\right)
\left( \begin{array}{c|c|c}
A & x & B \\
\hline
y & \mu & z \\
\hline
C & w & D 
\end{array}
\right)
\left(
\begin{array}{c|c|c}
I & 0 & 0 \\
\hline
0 & 1 & 0
\\
\hline
0 &  0 & X J
\end{array}
\right)
=
\left( \begin{array}{c|c|c}
A & x & B X J \\
\hline
y & \mu & z X J\\
\hline
J X^{-1} C & J X^{-1} w & J X^{-1} D X J 
\end{array}
\right)
\]
Observe that by the same computation as before we have that if $XA=DX$ then $J A = ( J X^{-1} D X J ) J$ and if $XBX=C$ then $J (BX J) J = J X^{-1} C$. But also, we have
\[
w=Xx \implies x = X^{-1} w \implies x = J (J X^{-1} w) \quad \text{ and } \quad
y=zX \implies y= (z X J)J.
\]
Hence, by Lemma~\ref{le:char_centrosym}, the matrix $Q^{-1}M Q$ is centrosymmetric.
\end{proof}

\vskip0.2cm

{\bf Remark:} Observe that if $\K$ is the field of complex numbers, then in the theorem above, if $X$ is a unitary matrix, it follows that the matrix $Q$ in the proof is also unitary and hence $M$ is unitarily equivalent to a centrosymmetric matrix.

\vskip0.2cm

It is natural to ask if the conditions in the above theorem are necessary. They are not. For example, if $\K$ is the field of complex numbers, consider the $2 \times 2$ matrix
\[
M=\begin{pmatrix}
1 & 3 \\
2 & 2
\end{pmatrix},
\]
for which there is no $1\times 1$ invertible matrix $X$ satisfying $XA=DX$ and $C=XBX$. However, $M$ is similar to the centrosymmetric matrix
\[
\begin{pmatrix}
\frac{3}{2} & \frac{5}{2} \\[3pt]
\frac{5}{2} & \frac{3}{2} 
\end{pmatrix},
\]
as can be easily seen, for example by noticing that the set of eigenvalues of both matrices is $\{ -1, 4\}$.

In fact, it is not hard to show that if $\K$ is the field of complex numbers, a $2 \times 2$ matrix is similar to a centrosymmetric matrix if and only if it is either a multiple of the identity or it has two distinct eigenvalues.

The following two examples show an application of Theorem~\ref{th:sim_cen} to certain Toeplitz matrices.

\begin{example}
Let $\alpha$ be a complex number and let $M$ be the $4\times 4$ Toeplitz matrix
\[
M=\begin{pmatrix}
\alpha & & \alpha -1 & &\alpha-2 & &\alpha-3 \\ 
\alpha+1 & & \alpha & &\alpha-1 & &\alpha-2 \\ 
\alpha+2 & &\alpha +1 & & \alpha & &\alpha-1 \\ 
\alpha+3 & &\alpha+2 & &\alpha+1 & &\alpha
\end{pmatrix}
\]
Let $A$, $B$, $C$ and $D$ be the $2 \times 2$ matrices
\[
A=D=\begin{pmatrix}
\alpha & & \alpha -1 \\
\alpha+1 & & \alpha
\end{pmatrix},
\quad
B=\begin{pmatrix}
\alpha-2 & &\alpha-3 \\ 
\alpha-1 & &\alpha-2
\end{pmatrix},
\quad \text{ and } \quad
C=\begin{pmatrix}
\alpha+2 & &\alpha +1 \\ 
\alpha+3 & &\alpha+2
\end{pmatrix}.
\]
If $\alpha \neq \pm \sqrt{5}$, we define $X$ to be the $2 \times 2$ matrix
\[
X=\frac{1}{\sqrt{\alpha^2-5}} \begin{pmatrix}
2 & \alpha -1 \\
\alpha+1 & 2
\end{pmatrix}.
\]
It is straightforward to check that $XA=DX$, that $C=X B X$ and that $X$ is invertible. Hence, by Theorem~\ref{th:sim_cen}, the matrix $M$ is similar to a centrosymmetric matrix. 

It is interesting to notice that if $\alpha=\pm \sqrt{5}$ the only two solutions to the matrix equation $C=X B X$ do not satisfy $XA=DX$ and hence our method does not answer the question of similarity of $M$ to a centrosymmetric matrix in this case.
\end{example}

\begin{example}
Let $\alpha$ be a complex number and let $M$ be the $6\times 6$ Toeplitz matrix
\[
M=\begin{pmatrix}
\alpha & & \alpha -1 & &\alpha-2 & &\alpha-3 & &\alpha-4 & &\alpha-5 \\ 
\alpha+1 & & \alpha & &\alpha-1 & &\alpha-2 & &\alpha-3 & &\alpha-4 \\ 
\alpha+2 & &\alpha +1 & & \alpha & &\alpha-1 & &\alpha-2 & &\alpha-3 \\ 
\alpha+3 & &\alpha+2 & &\alpha+1 & &\alpha & &\alpha-1 & &\alpha-2 \\
\alpha+4 & &\alpha+3 & &\alpha+2 & &\alpha+1 & &\alpha & &\alpha-1\\
\alpha+5 & &\alpha+4 & &\alpha+3 & &\alpha+2 & &\alpha+1 & &\alpha
\end{pmatrix}
\]
Let $A$, $B$, $C$ and $D$ be the $3 \times 3$ matrices
\[
A=D=\begin{pmatrix}
\alpha & & \alpha -1 & & \alpha-2 \\
\alpha+1 & & \alpha & & \alpha -1 \\
\alpha+2 & & \alpha+1 & & \alpha
\end{pmatrix},
\quad
B=\begin{pmatrix}
\alpha-3 & & \alpha-4 & & \alpha-5 \\ 
\alpha-2 & & \alpha-3 & & \alpha-4 \\ 
\alpha-1 & & \alpha-2 & & \alpha-3
\end{pmatrix},
\quad \text{ and } \quad
C=\begin{pmatrix}
\alpha+3 & & \alpha+2 & &\alpha +1 \\ 
\alpha+4 & & \alpha+3 & &\alpha +2 \\ 
\alpha+5 & & \alpha+4 & &\alpha +3
\end{pmatrix}.
\]
If $\alpha \neq \pm \sqrt{\frac{35}{3}}$, we define $X$ to be the $3 \times 3$ matrix
\[ 
X=\frac{1}{\sqrt{9 \alpha^2-105}} \begin{pmatrix}
0 & 16 & 3\alpha-13\\
20 & 3 (\alpha-9) & 16 \\
3\alpha -5 & 20 & 0
\end{pmatrix}.
\]
It is straightforward to check that $XA=DX$, that $C=X B X$ and that $X$ is invertible if $\alpha \neq 15$. Hence, by Theorem~\ref{th:sim_cen}, the matrix $M$ is similar to a centrosymmetric matrix in this case. If $\alpha=15$, we can use instead the invertible matrix
\[
X=\frac{1}{16 \sqrt{30}}
\begin{pmatrix}
-9 & & 50 & & 55 \\
58 & & 0 & & 50 \\
71 & & 58 & & -9
\end{pmatrix},
\]
which shows that $M$ is similar to a centrosymmetric matrix in this case as well.

It is interesting to notice that if $\alpha=\pm \sqrt{\frac{35}{3}}$ the only two solutions to the matrix equation $C=X B X$ do not satisfy $XA=DX$ and hence our method does not answer the question of similarity of $M$ to a centrosymmetric matrix in this case, as before.
\end{example}

The examples above show that the Toeplitz matrices $M$ of sizes $4\times 4$ and $6\times 6$ are similar to centrosymmetric matrices for all values of $\alpha$, except for perhaps two values in each case. It is easy to check that in the $2\times 2$ case, the matrix $M$ is similar to a centrosymetric matrix, by Theorem~\ref{th:sim_cen}, except for the values $\alpha=\pm 1$; in fact, for these values of $\alpha$ the matrix $M$ is not similar to a centrosymmetric matrix. 
It would be interesting to know if  for all sizes of the matrix $M$, it is similar to a centrosymmetric matrix for all values of $\alpha$ except for two. We leave this question open for future research.

\vskip0.5cm

In the next theorem, we generalize Theorem~\ref{th:sim_cen} to the case where $X$ is not necessarily invertible or even a square matrix.

\begin{theorem}
Let $n \in \N$ and let $M$ be an $n \times n$ matrix.
Assume $M$ is written as a block matrix of the form
    \[
M= \left( \begin{array}{c|c}
A & B \\
\hline 
C & D 
\end{array}
\right),
\]
where $A$ is an $s \times s$ matrix, $B$ is an $s \times (n-s)$ matrix, $C$ is an $(n-s) \times s$ matrix and $D$ is an $(n-s) \times (n-s)$ matrix, where $1 \leq s <n$. If there exists an $(n-s)  \times s$ matrix $X$, of rank $r>0$ such that $XA=DX$ and $C=XBX$, then $M$ is similar to a matrix containing a centrosymmetric principal submatrix of size $2r$. More precisely, $M$ is similar to a matrix of the form
\[
\left(\begin{array}{c|c}
U & V \\
\hline
W & Z
\end{array}\right)
\]
where $U$ is a centrosymmetric $(2r)\times (2r)$ matrix, $V$ is a $(2r) \times (n-2r)$ matrix, $W$ is a $(n-2r) \times (2r)$ matrix and $Z$ is a $(n-2r) \times (n-2r)$ matrix.
\end{theorem}
\begin{proof}
If $r=s=n-s$, then this theorem is just Theorem~\ref{th:sim_cen}.

Let us assume that $r<\min\{s,n-s\}$.  Since $X$ is of rank $r$ there exists an invertible $(n-s) \times (n-s)$ matrix $T$ and an invertible $s \times s$ matrix $S$ such that
\[
X':=T X S = \left( \begin{array}{c|c} 
I_{r} & 0 \\
\hline
0 & 0
\end{array} \right).
\]
Observe that
\[
\left(\begin{array}{c|c}
S^{-1} & 0 \\
\hline
0 & T
\end{array} \right)
\left(\begin{array}{c|c}
A & B \\
\hline
C & D
\end{array} \right)
\left(\begin{array}{c|c} S & 0 \\
\hline
0 & T^{-1}
\end{array} \right)
=
\left(\begin{array}{c|c}
S^{-1} A S & S^{-1} B T^{-1} \\
\hline
T C S & T D T^{-1}
\end{array} \right) 
=:
\left(\begin{array}{c|c}
A' & B' \\
\hline
C' & D'
\end{array} \right)=:M',
\]
where $A'$ is an $s \times s$ matrix, $B'$ is an $s \times (n-s)$ matrix, $C'$ is an $(n-s) \times s$ matrix and $D'$ is an $(n-s) \times (n-s)$ matrix. Hence $M$ is similar to $M'$.

We now write $A'$, $B'$, $C'$ and $D'$ as block matrices of the form
\[
A'= \left(\begin{array}{c|c}
A_{11}' & A_{12}' \\[2pt]
\hline
A_{21}' & A_{22}'
\end{array} \right), 
\quad
B'= \left(\begin{array}{c|c}
B_{11}' & B_{12}' \\[2pt]
\hline
B_{21}' & B_{22}'
\end{array} \right),
\quad
C'= \left(\begin{array}{c|c}
C_{11}' & C_{12}' \\[2pt]
\hline
C_{21}' & C_{22}'
\end{array} \right) 
\quad \text{ and } \quad
D'= \left(\begin{array}{c|c}
D_{11}' & D_{12}' \\[2pt]
\hline
D_{21}' & D_{22}'
\end{array} \right), 
\]
where $A_{11}'$, $B_{11}'$, $C_{11}'$ and $D_{11}'$ are $r \times r$ matrices (the sizes of the rest of the blocks can be easily determined but we will not write them since they will not be needed).

A computation shows that $XA=DX$ implies $X'A'=D'X'$ and that $C=XDX$ implies $C'=X'B'X'$. But, observe that
\[
\left( \begin{array}{c|c} 
A_{11}' & A_{12}' \\[2pt]
\hline
0 & 0
\end{array} \right)=
\left( \begin{array}{c|c} 
I_{r} & 0 \\
\hline
0 & 0
\end{array} \right)
\left(\begin{array}{c|c}
A_{11}' & A_{12}' \\[2pt]
\hline
A_{21}' & A_{22}'
\end{array} \right)
=X'A'=D'X'=
\left(\begin{array}{c|c}
D_{11}' & D_{12}' \\[2pt]
\hline
D_{21}' & D_{22}'
\end{array} \right)
\left( \begin{array}{c|c} 
I_{r} & 0 \\
\hline
0 & 0
\end{array} \right)=
\left( \begin{array}{c|c} 
D_{11}' & 0 \\[2pt]
\hline
D_{21}' & 0
\end{array} \right),
\]
and
\[
\left( \begin{array}{c|c} 
C_{11}' & C_{12}' \\[2pt]
\hline
C_{21}' & C_{22}'
\end{array} \right)=
C'= X' B' X'=
\left( \begin{array}{c|c} 
I_{r} & 0 \\
\hline
0 & 0
\end{array} \right)
\left(\begin{array}{c|c}
B_{11}' & B_{12}' \\[2pt]
\hline
B_{21}' & B_{22}'
\end{array} \right)
\left( \begin{array}{c|c} 
I_{r} & 0 \\
\hline
0 & 0
\end{array} \right)=
\left( \begin{array}{c|c} 
B_{11}' & 0 \\[2pt]
\hline
0 & 0
\end{array} \right),
\]
which implies that $A_{11}'=D_{11}'$ and $C_{11}'=B_{11}'$.

We now show that the matrix $M'$ is similar to a matrix of the form
\[
\left(\begin{array}{c|c}
U & V \\
\hline
W & Z
\end{array}\right)
\]
where $U$ is a centrosymmetric $(2r)\times (2r)$ matrix. Indeed, observe that
\[
\left( \begin{array}{c|c|c|c} 
I_r & 0 & 0 & 0  \\[2pt]
\hline
0 & 0 & J_{r} & 0 \\[2pt]
 \hline
0 & J_{s-r} & 0 & 0 \\[2pt]
\hline
0 & 0 & 0 & I_{n-s-r} \\[2pt]
\end{array} \right)
\left( \begin{array}{c|c|c|c} 
A_{11}' & A_{12}' & B_{11}' & B_{12}' \\[2pt]
\hline
A_{21}' & A_{22}' & B_{21}' & B_{22}' \\[2pt]
\hline
C_{11}' & C_{12}' & D_{11}' & D_{12}' \\[2pt]
\hline
C_{21}' & C_{22}' & D_{21}' & D_{22}' \\[2pt]
\end{array} \right)
\left( \begin{array}{c|c|c|c} 
I_r & 0 & 0 & 0  \\[2pt]
\hline
0 & 0 & J_{s-r} & 0 \\[2pt]
 \hline
0 & J_{r} & 0 & 0 \\[2pt]
\hline
0 & 0 & 0 & I_{n-s-r} \\[2pt]
\end{array} \right)
=
\left( \begin{array}{c|c|c|c} 
A_{11}' & B_{11}' J_{r} & A_{12}' J_{s-r} & B_{12}' \\[2pt]
\hline
J_{r} C_{11}' & J_{r} D_{11}' J_{r} & J_{r} C_{12}' J_{s-r} & J_{r} D_{12}' \\[2pt]
\hline
J_{s-r} A_{21}' & J_{s-r} B_{21}' J_{r} & J_{s-r} A_{22}' J_{s-r} & J_{s-r} B_{22}' \\[2pt]
\hline
C_{21}' &  D_{21}' J_{r} & C_{22}' J_{s-r} & D_{22}' \\[2pt]
\end{array} \right),
\]
and that the upper-left $2\times 2$ block in the above matrix, namely,
\[
U:=\left( \begin{array}{c|c} 
A_{11}' & B_{11}' J_{r}  \\[2pt]
\hline
J_r C_{11}' & J_r D_{11}' J_{r} 
\end{array}\right),
\]
is a centrosymmetric matrix, since $A_{11}'=D_{11}'$ and $C_{11}'=B_{11}'$. 

The proofs for each of the cases where $r=s <n-s$ and $r=n-s<s$ are similar and we omit them.
\end{proof}

In the following theorem we observe that if $X$ is a not necessarily square matrix, but is of full rank, then $M$ can be dilated to a matrix which is similar to a centrosymmetric matrix.

\begin{theorem}
Let $n \in \N$ and let $M$ be an $n \times n$ matrix.
Assume $M$ is written as a block matrix of the form
    \[
M= \left( \begin{array}{c|c}
A & B \\
\hline 
C & D 
\end{array}
\right),
\]
where $A$ is an $s \times s$ matrix, $B$ is an $s \times (n-s)$ matrix, $C$ is an $(n-s) \times s$ matrix and $D$ is an $(n-s) \times (n-s)$ matrix, where $1 \leq s <n$. If there exists an $(n-s)  \times s$ matrix $X$, of rank $\min\{s,n-s\}$ such that $XA=DX$ and $C=XBX$, then the matrix $M$ is a principal submatrix of a matrix $\widehat{M}$, which is similar to a centrosymmetric matrix. More precisely, if $k=\max\{s,n-s\}$,
there exist an $n\times(2k-n)$ matrix $B'$, a $(2k-n)\times n$ matrix $C'$ and a $(2k-n) \times (2k-n)$ matrix $D'$ such that the $(2k)\times (2k)$ matrix
\[
\widehat{M}= \left( \begin{array}{c|c}
M & B' \\
\hline 
C' & D'
\end{array}
\right)
\]
is similar to a centrosymmetric matrix.
\end{theorem}
\begin{proof}
If $s=\frac{n}{2}$, then this theorem is just Theorem~\ref{th:sim_cen}.

We first assume that $\frac{n}{2}<s$. In this case $k=s$. Since $X$ is an $(n-s)\times s$ matrix of rank $n-s$, there exists an $(2s-n)\times s$ matrix $Y$ such that
\[
\widehat{X}=\left(\begin{array}{c}
X\\
\hline
Y
\end{array}
\right)
\]
is an invertible $s \times s$ matrix. (Observe that $Y$ must be of rank $2s-n$.)

Let us rewrite the $(2s-n) \times s$ matrix $Y A \widehat{X}^{-1}$ as
\[
\left( \begin{array}{c|c} D_{21} & D_{22}\end{array} \right),
\]
where $D_{21}$ is a $(2s-n)\times (n-s)$ matrix and $D_{22}$ is a $(2s-n)\times (2s-n)$ matrix.

Now, if we define the $s \times s$ matrix
\[
\widehat{D}=\left(\begin{array}{c|c}
D & 0 \\
\hline
D_{21} & D_{22}
\end{array}\right),
\]
then observe that
\[
\widehat{X}A=\left(\begin{array}{c}
X\\
\hline
Y
\end{array}
\right) A = \left(\begin{array}{c}
X A\\
\hline
Y A
\end{array}
\right)
=
\left(\begin{array}{c}
D X\\
\hline
D_{21} X + D_{22}Y
\end{array}
\right) 
= \left(\begin{array}{c|c}
D & 0 \\
\hline
D_{21} & D_{22}
\end{array}\right) 
\left(\begin{array}{c}
X\\
\hline
Y
\end{array} \right) 
= \widehat{D} \widehat{X},
\]
where the third equality follows from the hypothesis $XA=DX$ and from $YA=\left(\begin{array}{c|c} D_{21} & D_{22} \end{array} \right) \widehat{X}$.

Now, we choose an $s \times (2s-n)$ matrix $B_2$ such that $X B_2=0$ and we define a $(2s-n) \times s$ matrix $C_2$ as $C_2=Y B X+Y B_2 Y$. Then, if we define the $s \times s$ matrices $\widehat{B}$ and $\widehat{C}$ as
\[
\widehat{B} = \left(\begin{array}{c|c} B & B_2 \end{array} \right) 
\quad \text{ and } \quad
\widehat{C} = \left(\begin{array}{c} C \\
\hline C_2 \end{array} \right),
\]
we can see that
\[
\widehat{X} \ \widehat{B} \ \widehat{X} 
= \left(\begin{array}{c} X\\ \hline Y \end{array} \right) 
\left(\begin{array}{c|c} B & B_2 \end{array} \right) 
\left(\begin{array}{c} X\\ \hline Y \end{array} \right)
=\left( \begin{array}{c}
X B X + X B_2 Y \\
\hline
Y B X + Y B_2 Y
\end{array} \right)
=\left( \begin{array}{c}
C\\
\hline
C_2
\end{array} \right)
=\widehat{C},
\]
where the third equality follows from the hypothesis $XBX=C$ and the choice of matrices $B_2$ and $C_2$. 

Therefore, by Theorem~\ref{th:sim_cen}, the matrix
\[
\widehat{M}=
\left(
\begin{array}{c|c}
    A & \widehat{B} \\[2pt]
    \hline \\[-10pt]
    \widehat{C} & \widehat{D} 
\end{array} \right) =
\left(
\begin{array}{c|c|c}
    A & B & B_2 \\
    \hline 
    C & D & 0 \\
    \hline
    C_2 & D_{21} & D_{22}
\end{array}
\right) 
= \left(
\begin{array}{c|c}
    M & B' \\
    \hline 
    C' & D'
\end{array} \right) 
\]
is similar to a centrosymmetric matrix, which finishes the case $\frac{n}{2} < s$.

Now we assume that $1 \leq s < \frac{n}{2}$. In this case $k=n-s$. Since $X$ is of rank $s$ there exists an $(n-s) \times (n-2s)$ matrix $Y$ such that
\[
\widehat{X}=\left(\begin{array}{c|c}
Y & X
\end{array} \right)
\]
is an invertible $(n-s) \times (n-s)$ matrix. (Observe that $Y$ must be of rank $n-2s$.) Now, proceeding in a similar fashion as above, we can choose an $(n-2s) \times (n-2s)$ matrix $A_{11}$, an $s \times(n-2s)$ matrix $A_{21}$ such that
\[
\left(\begin{array}{c}
A_{11} \\
\hline 
A_{21}
\end{array} \right) = \widehat{X}^{-1} D Y,
\]
 an $(n-2s) \times (n-s)$ matrix $B_1$ such that $B_1 X=0$, and we can set the $(n-s)\times (n-2s)$ matrix  $C_1$ as $C_1=Y B_1 Y + X B Y$. Then the $(n-s) \times (n-s)$ matrices
\[
\widehat{A}=\left(\begin{array}{c|c}
A_{11} & 0 \\
\hline
A_{21} & A
\end{array}\right), 
\qquad 
\widehat{B}=\left(\begin{array}{c}
B_1 \\
\hline
B
\end{array}\right), 
\quad \text{ and } \quad
\widehat{C}=\left(\begin{array}{c|c}
C_1 & C
\end{array}\right)
\]
satisfy the equations
\[
\widehat{X} \widehat{A} = D \widehat{X}, \quad \text{ and } \quad \widehat{C}= \widehat{X} \ \widehat{B} \ \widehat{X}.
\]
Therefore, by Theorem~\ref{th:sim_cen}, the $2(n-s) \times 2(n-s)$ matrix
\[
\widehat{M}=
\left(
\begin{array}{c|c}
    \widehat{A} & \widehat{B} \\[2pt]
    \hline \\[-10pt]
    \widehat{C} & D 
\end{array} \right) =
\left(
\begin{array}{c|c|c}
    A_{11} & 0 & B_1 \\
    \hline 
    A_{21} & A & B \\
    \hline
    C_1 & C & D
\end{array}
\right) 
= \left(
\begin{array}{c|c}
    D' & C' \\
    \hline 
    B' & M
\end{array} \right) 
\]
is similar to a centrosymmetric matrix. Finally, since
\[
\left(\begin{array}{c|c}
    M & B' \\
    \hline 
    C' & D'
\end{array} \right) 
=
\left(\begin{array}{c|c}
    0 & I_n \\
    \hline 
    I_{n-2s} & 0
\end{array} \right) 
\left(\begin{array}{c|c}
    D' & C' \\
    \hline 
    B' & M
\end{array} \right) 
\left(\begin{array}{c|c}
    0 & I_{n-2s} \\
    \hline 
    I_n & 0
\end{array} \right),
\]
the result follows.
\end{proof}

{\bf Remark}: If $\K$ is the field of complex numbers, observe that in the proof above, if the rows of the matrix $X$ are orthonormal, then the matrix $\widehat{X}$ can be chosen to be unitary. In this case, $\widehat{M}$ will be unitarily equivalent to a centrosymmetric matrix (by the remark following Theorem~\ref{th:sim_cen}).


\section{Factorization of determinants}\label{se:fact}

One of the interesting facts about centrosymmetric matrices is that the determinant of a centrosymmetric matrix can be factored as the product of the determinants of two particular matrices, as described in the following well-known theorem (a proof can be found in, for example \cite[p. 125]{aitken}).

\begin{theorem}\label{th:det_censym}
Let $M$ be an $n\times n$ centrosymmetric matrix.
\begin{itemize}
    \item Assume $n$ is even and $M$ is written as
    \[
M= \left( \begin{array}{c|c}
A & B \\
\hline 
C & D 
\end{array}
\right),
\]
where $A$, $B$, $C$ and $D$ are $\frac{n}{2} \times \frac{n}{2}$ matrices. Then $\det(M)=\det(A+BJ) \det(A-B J)$.
\item  Assume $n$ is odd and $M$ is written as
\[
M= \left( \begin{array}{c|c|c}
A & x & B \\
\hline
y & \mu & z \\
\hline
C & w & D 
\end{array}
\right)
\]
where $A,B,C,D$ are all $\frac{n-1}{2} \times \frac{n-1}{2}$ matrices, where $x, w$ are $\frac{n-1}{2} \times 1$ matrices, where $y, z$ are $1 \times \frac{n-1}{2}$ matrices, and $\mu \in \K$. Then,
\[
\det(M)=
\det \left( \begin{array}{c|c}
A + B J & x \\
\hline
2 y & \mu 
\end{array}
\right) \det(A -B J).
\]
\end{itemize}
\end{theorem}

In view of the results in the previous section, it is natural to ask if a block matrix 
\[
M= \left( \begin{array}{c|c}
A & B \\
\hline 
C & D 
\end{array}
\right),
\]
satisfying the equations $XA=DX$ and $C=XBX$ for some matrix $X$, has a corresponding factorization. In fact, the following theorem, which establishes a factorization of the determinant under more general conditions, is well-known.

\begin{theorem}\label{th:main}
Let $n \in \N$ and let $M$ be an $n \times n$ matrix.
Assume $M$ is written as a block matrix of the form
    \[
M= \left( \begin{array}{c|c}
A & B \\
\hline 
C & D 
\end{array}
\right),
\]
where $A$ is an $s \times s$ matrix, $B$ is an $s \times (n-s)$ matrix, $C$ is an $(n-s) \times s$ matrix and $D$ is an $(n-s) \times (n-s)$ matrix. If there exists an $(n-s)  \times s$ matrix $X$ such that $C=XA-DX+XBX$, then
\[
\det(M)=\det(A+BX) \det(D-X B).
\]
\end{theorem}
\begin{proof}
The result will follow by performing the multiplication:
\[
\left( \begin{array}{c|c}
I_s & 0 \\
\hline
-X & I_{n-s} 
\end{array}
\right)
\left( \begin{array}{c|c}
A & B \\
\hline
C & D 
\end{array}
\right)
 \left( \begin{array}{c|c}
I_s & 0 \\
\hline
X & I_{n-s} 
\end{array}
\right)
=
 \left( \begin{array}{c|c}
A + B X & B \\
\hline
C-XA+DX -XBX & D- X B  
\end{array}
\right)
=
 \left( \begin{array}{c|c}
A + B X & B \\
\hline
0 & D-X B 
\end{array}
\right),
\]
where the last equality follows by the hypothesis. Taking determinants on both sides, the result follows.
\end{proof}

This theorem can also be found, for example, in \cite{Meyer} or \cite{NoTa}.

In view of Theorem~\ref{th:main}, and recalling Theorem~\ref{th:sim_cen}, it is natural to ask if the existence of an invertible matrix $X$ such that $C=XA-DX+XBX$ implies that $M$ is similar to a centrosymmetric matrix. However, this is false,  as the next example shows. 

\begin{example}Assume $\K$ is a field of characteristic not equal to $2$. The $2 \times 2$ matrix
\[
M=\begin{pmatrix}
1 & -1 \\ 
1 & -1
\end{pmatrix}
\]
satisfies the equations $C=XA-DX+XBX$ for $X=1$, but it is not similar to a centrosymmetric matrix. Indeed, $M^2=0$ and hence every matrix $Q$ similar to $M$ must satisfy the equation $Q^2=0$. But it is easy to see that the only $2 \times 2$ centrosymmetric matrix $Q$ which satisfies $Q^2=0$ is the zero matrix. But clearly $M$ is not similar to the zero matrix.
\end{example}

Observe that this example also shows that the equations $C=XA$ and $DX=XBX$, or the equations $C=-DX$ and $XA=-XBX$ do not imply that $M$ is similar to a centrosymmetric matrix.

Similarly to Theorem~\ref{th:main}, it is possible to show the following theorem, which is also well-known.

\begin{theorem}\label{th:main_prime}
Let $M$ be an $n \times n$ matrix.
Assume $M$ is written as a block matrix of the form
\[
M= \left( \begin{array}{c|c}
A & B \\
\hline 
C & D 
\end{array}
\right),
\]
where $A$ is an $s \times s$ matrix, $B$ is an $s \times (n-s)$ matrix, $C$ is an $(n-s) \times s$ matrix and $D$ is an $(n-s) \times (n-s)$ matrix. If there exists an $s\times (n-s)$ matrix $Y$ such that $B=YD-AY+YCY$, then
\[
\det(M)=\det(A-YC) \det(D+CY).
\]
\end{theorem}
\begin{proof}
The proof is similar to the proof of Theorem~\ref{th:main}. Just observe that
\[
\left( \begin{array}{c|c}
I_s & -Y \\
\hline
0& I_{n-s} 
\end{array}
\right)
\left( \begin{array}{c|c}
A & B \\
\hline 
C & D 
\end{array}
\right)
\left( \begin{array}{c|c}
I_s & Y \\
\hline
0 & I_{n-s} 
\end{array}
\right)
= \left( \begin{array}{c|c}
A-YC & 0 \\
\hline 
C & D +CY
\end{array}
\right). \qedhere
\]
\end{proof}

Nevertheless, the following question arises. Assume that both equations
\[
C=XA-DX+XBX, \quad \text{ and }\quad B=YD-AY+YCY
\]
are satisfied: are the factorizations given by Theorems~\ref{th:main} and \ref{th:main_prime} the same?

In case $X$ is invertible (here we asumme $n$ is even and $s=\frac{n}{2}$) it immediately follows that $B=X^{-1}D - A X^{-1}+ X^{-1}C X^{-1}$ and hence $B=YD-AY+YCY$ is satisfied with $Y=X^{-1}$. In this case the factorizations are the same. Indeed, observe that if $C=XA-DX+XBX$, then Theorem~\ref{th:main} gives that
\[
\det(M)=\det(A+BX) \det(D-XB).
\]
Since $B=X^{-1}D - A X^{-1}+ X^{-1}C X^{-1}$, Theorem~\ref{th:main_prime} gives that
\[
\det(M)=\det(A-X^{-1}C) \det( D + C X^{-1}).
\]
But observe that since $BX=X^{-1}DX-A+X^{-1}C$, then $A-X^{-1}C=(X^{-1}D-B)X$, and hence $A-X^{-1}C=X^{-1}(D-XB)X$, from which $\det(A-X^{-1}C)=\det(D-XB)$. Analogously, since $XB=D-XAX^{-1}+C X^{-1}$, then $D+CX^{-1}=XB+XAX^{-1}=X(A+BX)X^{-1}$, and hence $\det(D+CX^{-1})=\det(A+BX)$.  Therefore, in this case, the factorizations given by Theorems~\ref{th:main} and \ref{th:main_prime} are the same.

On the other hand, it is possible that $C=XA-DX+XBX$ and $B=YD-AY+YCY$ are satisfied but $X$ and $Y$ are not inverses of each other. Indeed, observe that if
\[
M=\left(\begin{array}{c|c}
I & I \\
\hline
0 & 2 I 
\end{array}\right),
\]
then $C=XA-DX+XBX$ with $X=0$ and $B=YD-AY+YCY$ with $Y=I$ and hence $X$ and $Y$ are not inverses of each other. But in this case,  the factorizations given by Theorems~\ref{th:main} and \ref{th:main_prime} are the same.

In general, regardless of the parity of $n$, is it possible that both equations
\[
C=XA-DX+XBX, \quad \text{ and }\quad B=YD-AY+YCY
\]
are satisfied and the factorizations given by Theorems~\ref{th:main} and \ref{th:main_prime} are not the same? We leave this question open for future research.

One way to obtain examples for Theorem~\ref{th:main} is by solving a system of matrix equations like the ones in the first section of this paper: $XA=DX$ and $C=XBX$, instead of the equation  $C=XA - DX + XBX$. Alternatively, one may require the existence of a matrix $X$ such that $C=XA$ and $DX=XBX$. If this case occurs and $X \neq 0$, the determinant of $M$ turns out to be zero. The same happens with other pairs of equations, as the next proposition shows.

\begin{proposition}
Let $M$ be an $n \times n$ matrix.
Assume $M$ is written as a block matrix of the form
    \[
M= \left( \begin{array}{c|c}
A & B \\
\hline 
C & D 
\end{array}
\right),
\]
where $A$ is an $s \times s$ matrix, $B$ is an $s \times (n-s)$ matrix, $C$ is an $(n-s) \times s$ matrix and $D$ is an $(n-s) \times (n-s)$ matrix. Assume there exists an $(n-s) \times s$ matrix $X\neq 0$ or an $s \times (n-s)$ matrix $Y\neq 0$ such that one of the folowing four systems of equations hold:
\begin{enumerate}
    \item $C=XA$ and $DX=XBX$;
    \item $C=-DX$ and $XA=-XBX$;
    \item $B=YD$ and $AY=YCY$;
    \item $B=-AY$ and $YD=-YCY$.
\end{enumerate}
Then $\det(M)=0$.
\end{proposition}
\begin{proof}
In the first case, observe that since $(D-XB)X=0$ and $X\neq 0$, then $D-XB$ is singular and hence $\det(D-XB)=0$. Applying Theorem \ref{th:main}, we obtain the desired result.

In the second case, observe that since $X(A+BX)=0$ and $X\neq 0$, then $A+BX$ is singular and hence $\det(A+BX)=0$. Again, applying Theorem \ref{th:main}, we obtain the desired result.
The rest of the cases are proved similarly.
\end{proof}

\section{Two useful examples}\label{se:exa}

As an application of some of the previous results, we obtain the following two co\-ro\-lla\-ries, which we needed to obtain the results in \cite{IMN2}.

\begin{corollary}\label{cor:Apm}
Let $n \in \N$ with $n \geq 2$ 
and consider the two $(n+1)\times (n+1)$ matrices
\begin{equation*}
\Atc=\begin{pmatrix}
t & \ce_0 & 0 & 0 & \cdots &0  & 0 & \pm \ce_n \\
\ce_0 & t & \ce_1 & 0 & \cdots &0  & 0 & 0 \\
0 & \ce_1 & t & \ce_2 & \cdots &0  & 0 & 0 \\
0 & 0 & \ce_2 & t & \cdots &0  & 0 & 0 \\
\vdots & \vdots & \vdots & \vdots & \ddots & \vdots & \vdots & \vdots \\
0 & 0 & 0 & 0  & \cdots & t & \ce_{n-2} & 0 \\
0 & 0 & 0 & 0 & \cdots & \ce_{n-2 }& t & \ce_{n-1} \\
\pm\ce_n & 0 & 0 & 0  & \cdots & 0 & \ce_{n-1} & t \end{pmatrix}.
\end{equation*}
\begin{itemize} 
\item If $n+1$ is odd, assume that $\ce_j=\ce_{n-j+1}$ for $j=1, 2, \dots, \frac{n}{2}$ (we make no assumption on $\ce_0$).
Then the determinant of $\Atc$ equals the product
\begin{equation*}
\det\begin{pmatrix}
t\pm \ce_0 & \ce_1 & 0 & \dots & 0 & 0 & 0 \\
\ce_1 & t & \ce_2 & \dots & 0 & 0 & 0 &  \\
0 & \ce_2 & t & \dots  & 0  & 0 & 0  \\
\vdots & \vdots & \vdots & \ddots & \vdots & \vdots & \vdots \\
0 & 0 & 0 & \dots & t & \ce_{\frac{n}{2}-1} & 0 & \\
0 & 0 & 0 & \dots & \ce_{\frac{n}{2}-1} & t  &   \ce_{\frac{n}{2}} \\
0 & 0 & 0 & \dots & 0 & 2 \ce_{\frac{n}{2}} & t \\
\end{pmatrix}
\det\begin{pmatrix}
t\mp \ce_0 &\ce_{1} & 0 & \cdots & 0 & 0 \\
\ce_{1} & t & \ce_{2} & \cdots  & 0 & 0 \\
0 &  \ce_{2} & t & \cdots  & 0 & 0 \\
\vdots & \vdots & \vdots & \ddots & \vdots & \vdots \\
0 & 0 & 0 & \cdots & t & \ce_{\frac{n}{2}-1} \\
0 & 0 & 0 & \cdots & \ce_{\frac{n}{2}-1} & t 
\end{pmatrix}.
\end{equation*}
Observe that, if $n+1=3$, the matrix in the determinant in the right-hand-side of the expression above is the $1\times 1$ matrix $(t \mp \ce_0)$. 

\item If $n+1$ is even, assume that $\ce_j=\ce_{n-j+1}$ for $j=1, 2, \dots, \frac{n-1}{2}$ (we make no assumption on $\ce_0$, nor on $\ce_{\frac{n+1}{2}}$). 
Then the determinant of $\Atc$ equals the product
\begin{equation*}
    \det\begin{pmatrix}
t\pm\ce_0 & \ce_1 & 0 & \dots & 0 & 0 & 0 \\
\ce_1 & t & \ce_2 & \dots & 0 & 0 & 0 &  \\
0 & \ce_2
& t & \dots  & 0  & 0 & 0  \\
\vdots & \vdots & \vdots & \ddots & \vdots & \vdots & \vdots \\
0 & 0 & 0 & \dots & t & \ce_{\frac{n+1}{2}-2} & 0 & \\
0 & 0 & 0 & \dots & \ce_{\frac{n+1}{2}-2} & t  &   \ce_{\frac{n+1}{2}-1} \\
0 & 0 & 0 & \dots & 0 &  \ce_{\frac{n+1}{2}-1} & t+\ce_{\frac{n+1}{2}} \\
\end{pmatrix}
\det\begin{pmatrix}
t\mp\ce_0 & \ce_{1} & 0 & \dots & 0 & 0 & 0 \\
\ce_{1} & t & \ce_2 & \dots & 0 & 0 & 0 &  \\
0 & \ce_2
& t & \dots  & 0  & 0 & 0  \\
\vdots & \vdots & \vdots & \ddots & \vdots & \vdots & \vdots \\
0 & 0 & 0 & \dots & t & \ce_{\frac{n+1}{2}-2} & 0 & \\
0 & 0 & 0 & \dots & \ce_{\frac{n+1}{2}-2} & t  &   \ce_{\frac{n+1}{2}-1} \\
0 & 0 & 0 & \dots & 0 &  \ce_{\frac{n+1}{2}} & t-\ce_{\frac{n+1}{2}}
\end{pmatrix}.
\end{equation*}
\end{itemize}
\end{corollary}
\begin{proof}
We divide the proof in two cases.

\vskip0.2cm

\noindent {\em Case $\Atcp$:} First of all, let $Q$ be the $(n+1)\times (n+1)$ matrix:
\[
Q=\begin{pmatrix}
0 & 1 & 0 & 0 & \dots & 0 & 0 \\
0 & 0 & 1 & 0 & \dots & 0 & 0 \\
0 & 0 & 0 & 1 &  \dots & 0 & 0\\
0 & 0 & 0 & 0 &  \dots & 0 & 0\\
\vdots & \vdots & \vdots & \vdots & \ddots & \vdots & \vdots \\
0 & 0 & 0 & 0 & \dots & 0 & 1 \\
1 & 0 & 0 & 0 &  \dots & 0 & 0\\
\end{pmatrix}.
\]
It is straightforward to check that
\[
Q \Atcp Q^{-1} = \begin{pmatrix}
t & \ce_1 & 0 & 0 & \cdots &0  & 0 & \ce_0 \\
\ce_1 & t & \ce_2 & 0 & \cdots &0  & 0 & 0 \\
0 & \ce_2 & t & \ce_3 & \cdots &0  & 0 & 0 \\
0 & 0 & \ce_3 & t & \cdots &0  & 0 & 0 \\
\vdots & \vdots & \vdots & \vdots & \ddots & \vdots & \vdots & \vdots \\
0 & 0 & 0 & 0 & \cdots & t & \ce_{n-1} & 0 \\
0 & 0 & 0 & 0 & \cdots & \ce_{n-1 }& t & \ce_{n} \\
\ce_0 & 0 & 0 & 0  & \cdots & 0 & \ce_{n} & t 
\end{pmatrix},
\]
and hence $\det(Q \Atcp Q^{-1})= \det(\Atcp)$. 

By the hypothesis, $Q \Atcp Q^{-1}$ is centrosymmetric. We now apply Theorem~\ref{th:det_censym} to obtain the desired result.

\noindent {\em Case $\Atcm$:} In this case, let $Q$ be the $(n+1)\times (n+1)$ matrix:
\[
Q=\begin{pmatrix}
0 & 1 & 0 & 0 & \dots & 0 & 0 \\
0 & 0 & 1 & 0 & \dots & 0 & 0 \\
0 & 0 & 0 & 1 &  \dots & 0 & 0\\
0 & 0 & 0 & 0 &  \dots & 0 & 0\\
\vdots & \vdots & \vdots & \vdots & \ddots & \vdots & \vdots \\
0 & 0 & 0 & 0 & \dots & 0 & 1 \\
-1 & 0 & 0 & 0 &  \dots & 0 & 0\\
\end{pmatrix}.
\]
It is straightforward to check that
\[
Q \Atcm Q^{-1} = \begin{pmatrix}
t & \ce_1 & 0 & 0 & \cdots &0  & 0 & -\ce_0 \\
\ce_1 & t & \ce_2 & 0 & \cdots &0  & 0 & 0 \\
0 & \ce_2 & t & \ce_3 & \cdots &0  & 0 & 0 \\
0 & 0 & \ce_3 & t & \cdots &0  & 0 & 0 \\
\vdots & \vdots & \vdots & \vdots & \ddots & \vdots & \vdots & \vdots \\
0 & 0 & 0 & 0 & \cdots & t & \ce_{n-1} & 0 \\
0 & 0 & 0 & 0 & \cdots & \ce_{n-1 }& t & \ce_{n} \\
-\ce_0 & 0 & 0 & 0 & \cdots & 0 & \ce_{n} & t 
\end{pmatrix},
\]
and hence $\det(Q \Atcm Q^{-1})= \det(\Atcm)$. 

By the hypothesis, $Q \Atcm Q^{-1}$ is centrosymmetric.  We now apply Theorem~\ref{th:det_censym} to obtain the desired result.
\end{proof}

The last result of this paper is now an immediate consequence of Theorem~\ref{th:det_censym}.

\begin{corollary}\label{cor:Bpm}
Let $n \in \N$ with $n \geq 2$
and consider the two $(n+1)\times (n+1)$ matrices
\begin{equation*}
B^{\pm}=\begin{pmatrix}
t\pm \ce_0 & \ce_1 & 0 & 0 &  \cdots &0  & 0 & 0 \\
\ce_1 & t & \ce_2 & 0 & \cdots &0  & 0 & 0 \\
0 & \ce_2 & t & \ce_3 & \cdots &0  & 0 & 0 \\
0 & 0 & \ce_3 & t & \cdots &0  & 0 & 0 \\
\vdots & \vdots & \vdots & \vdots & \ddots & \vdots & \vdots & \vdots \\
0 & 0 & 0 & 0 &  \cdots & t & \ce_{n-1} & 0 \\
0 & 0 & 0 & 0 & \cdots & \ce_{n-1 }& t & \ce_{n} \\
0 & 0 & 0 & 0 & \cdots & 0 & \ce_{n} & t\pm \ce_0 
\end{pmatrix}.
\end{equation*}
\begin{itemize}
    \item If $n+1$ is odd, assume that $\ce_j=\ce_{n-j+1}$ for $j=1, 2, \dots, \frac{n}{2}$ (we make no assumption on $\ce_0$). Then the determinant of $B^\pm$ equals the product
\begin{equation*}
    \det\begin{pmatrix}
    t\pm\ce_0 & \ce_1 & 0 & \cdots & 0 & 0 \\
    \ce_1 & t & \ce_2 & \cdots & 0 & 0 \\
    0 & \ce_2 & t & \cdots & 0 & 0  \\
    \vdots & \vdots & \vdots & \ddots  &\vdots & \vdots \\
    0 & 0 & 0 & \cdots & t & \ce_{\frac{n}{2}-1} \\
    0 & 0 & 0 & \cdots &  \ce_{\frac{n}{2}-1} & t\\
    \end{pmatrix}
    \det\begin{pmatrix}
    t\pm\ce_0 & \ce_1 & 0 & \cdots & 0 & 0 & 0\\
    \ce_1 & t & \ce_2 & \cdots & 0 & 0 & 0\\
    0 & \ce_2 & t & \cdots & 0 & 0 & 0 \\
    \vdots & \vdots & \vdots & \ddots  &\vdots & \vdots & \vdots\\
    0 & 0 & 0 & \cdots & t & \ce_{\frac{n}{2}-1} & 0\\
    0 & 0 & 0 & \cdots &  \ce_{\frac{n}{2}-1} & t & \ce_\frac{n}{2}\\
    0 & 0 & 0 & \cdots & 0 & 2 \ce_{\frac{n}{2}} & t\\
    \end{pmatrix}.
\end{equation*}
Observe that, if $n+1=3$, the matrix in the determinant in the left-hand-side of the expression above is the $1\times 1$ matrix $(t \pm \ce_0)$. 

\item If $n+1$ is even, assume that $\ce_j=\ce_{n-j+1}$ for $j=1, 2, \dots, \frac{n-1}{2}$ (we make no assumption on $\ce_0$, nor on $\ce_{\frac{n+1}{2}}$). Then the determinant of $B^\pm$ equals the product
\[   
\det\begin{pmatrix}
    t\pm\ce_0 & \ce_1 & 0 & \cdots & 0 & 0 \\
    \ce_1 & t & \ce_2 & \cdots & 0 & 0 \\
    0 & \ce_2 & t & \cdots & 0 & 0  \\
    \vdots & \vdots & \vdots & \ddots  &\vdots & \vdots \\
    0 & 0 & 0 & \cdots & t & \ce_{\frac{n-1}{2}} \\
    0 & 0 & 0 & \cdots &  \ce_{\frac{n-1}{2}} & t+\ce_{\frac{n+1}{2}}\\
    \end{pmatrix}
        \det\begin{pmatrix}
    t\pm\ce_0 & \ce_1 & 0 & \cdots & 0 & 0 \\
    \ce_1 & t & \ce_2 & \cdots & 0 & 0 \\
    0 & \ce_2 & t & \cdots & 0 & 0  \\
    \vdots & \vdots & \vdots & \ddots  &\vdots & \vdots \\
    0 & 0 & 0 & \cdots & t & \ce_{\frac{n-1}{2}} \\
    0 & 0 & 0 & \cdots &  \ce_{\frac{n-1}{2}} & t- \ce_{\frac{n+1}{2}}\\
    \end{pmatrix}.
    \]
    \end{itemize}
\end{corollary}

\end{document}